\documentclass[11pt,reqno,oneside]{article}

\usepackage{amsmath, amsfonts, amssymb, latexsym, amsthm}

\usepackage{geometry}
\usepackage{mathrsfs}
\usepackage{graphicx} 
\usepackage{color}    
\usepackage{subfigure}
\usepackage{hyperref} 

\usepackage{full page}
\usepackage{fancyhdr} 

\newtheorem{theorem}{Theorem}
\theoremstyle{plain}

\newtheorem{lemma}[theorem]{Lemma}

\newtheorem{remark}[theorem]{Remark}

\numberwithin{equation}{section}
\numberwithin{theorem}{section}

\newcommand{\cU}{\mathcal{U}}

\newcommand{\cJ}{\mathcal{J}}

\newcommand{\cF}{\mathcal{F}}

\newcommand{\cW}{\mathcal{W}}
\newcommand{\cB}{\mathcal{B}}

\newcommand{\E}{\mathbb{E}}
\newcommand{\R}{\mathbb{R}}
\newcommand{\F}{\mathbb{F}}

\newcommand{\N}{\mathbb{N}}
\newcommand{\bP}{\mathbb{P}} 

\newcommand{\wt}{\widetilde}

\newcommand{\Ito}{It\^{o}}
\newcommand{\e}{\varepsilon}

\begin{document}

\title
{Optimal Actuator Location of the Minimum Norm Controls for Stochastic Heat Equations}
\author{Donghui Yang\thanks{School of Mathematics and Statistics,
    Central South University, Changsha, 410075, P.R. China (\tt{donghyang@139.com}).}\quad and\quad Jie Zhong\thanks{Department of Mathematics,
    University of Rochester, Rochester, NY 14627, USA (\tt{jiezhongmath@gmail.com}).}}

\maketitle

\begin{abstract}
In this paper, we study the approximate controllability for the stochastic
  heat equation over measurable sets, and the optimal actuator location of the
  minimum norm controls. We formulate a relaxed optimization problem for both
  actuator location and its corresponding minimum norm control into a two-person
  zero sum game problem and develop a sufficient and necessary condition for the
  optimal solution via Nash equilibrium. At last, we prove that the relaxed optimal solution is an optimal actuator location for the classical problem.
\end{abstract}

 \vspace{0.2cm}


\section{Introduction}
Let $\F = (\Omega, \cF,\{\cF_t\}_{t\ge 0}, \bP)$ be a stochastic
basis with usual conditions. On $\F$, we define a standard scalar Wiener
process $W = \{w(t)\}_{t\ge 0}$. For simplicity, we assume that the
filtration 
$\{\cF_t\}_{t\ge0}$ is generated by $W$.

Fix $T>0$.

Given a Hilbert space $H$, we denote by $L^2_\cF(0,T;H)$ the Banach
space consisting of all $H$-valued $\{\cF_t\}_{t\ge0}$-adapted
processes $X$ such that the square of the canonical norm  $\E\|X(\cdot)\|^2_{L^2(0,T;H)}
<\infty$; denote by
$L^\infty_\cF(0,T;H)$ the Banach space consisting of all $H$-valued
$\{\cF_t\}_{t\ge0}$-adapted bounded processes, with the essential
supremum norm; and denote by $L^2_\cF(\Omega;C([0,T];H))$ the Banach
space consisting of all $H$-valued $\{\cF_t\}_{t\ge0}$-adapted continuous
processes $X$ such that the square of the canonical norm $\E
\|X(\cdot)\|^2_{C(0,T;H)}<\infty$. For any $t\in[0,T]$, the space
$L^2(\Omega,\cF_t,\bP;H)$ consists of all $H$-valued
$\cF_t$-measurable random variables with finite second moments.

Let $D$ be a bounded domain in 
$\R^d$ with a $C^2$ boundary 
$\partial D$.  Let 
$G$ be measurable subset with positive measures of 
$D$. 

We denote by
$(\cdot, \cdot)$ the inner product in $L^2(D)$, and denote by
$\|\cdot\|$ the norm induced by $(\cdot,\cdot)$. We also use the
notations $(\cdot,\cdot)_G$ and $\|\cdot\|_G$ for the inner product
and the norm defined on $L^2(G)$, respectively. We denote by $|\cdot|$
the Lebesgue measure on $\R^d$.

Let $A$ be an unbounded linear operator on $L^2(D)$:
\[\mathcal{D}(A)=H^2(D)\cap H^1_0(D),\quad Av = \Delta v,\ \forall v\in
\mathcal{D}(A).\]

Consider the
following stochastic heat equation
\begin{equation}
  \label{eq:main}
  \begin{cases}
    dy = Ay\, dt + \chi_G u(t)\, dt + a(t)y\, dw(t),\quad t\in (0,T),\\
    y(0) = y_0,
  \end{cases}
\end{equation}
where $a\in L^\infty_\cF(0,T;\R)$.

 We say the system~\eqref{eq:main} is {\it approximately controllable} at time $T$, if for
  any initial data $y_0\in L^2(\Omega, \cF_0, \bP;L^2((D))$, and any final state
  $y_1\in L^2(\Omega, \cF_T, \bP;L^2(D))$ and any $\e>0$, there exists a control
  $u$ in the space $L_\cF^2(0,T;L^2(D))$ such that the solution $y$ of the
  system~\eqref{eq:main} with initial data $y_0$ and control $u$ satisfies
  \[
    \E \|y(T) - y_1\|^2 \le \e.
  \]

Without of generality, we simply choose $y_1=0$.

In the sequel, fix $\e>0$.

Our first result is to confirm the approximate controllability for system
\eqref{eq:main}. Moreover, we solve the following minimum norm control problem
\begin{equation}
  \label{eq:norm optimal}
  N_\e(G) = \inf\{ \E \|u\|^2_{L^2((0,T)\times D)} \mid \E \|y(T;G,u)\|^2 \le \e\},
\end{equation}
where $y(\cdot;G,u)$ is the solution of system~\eqref{eq:main}. In
the problem \eqref{eq:norm optimal}, we say $u$ is an {\it admissible
  control} if $u \in L^2_\cF(0,T;L^2(D))$ and $\E \|y(T;G,u)\|^2 \le \e$; we say
$u^\ast$ is a minimal norm control if $u^\ast$ is an admissible control such that $N_\e(G)$ is achieved. 

To present our result, let us introduce the following backward stochastic
  heat equation
  \begin{equation}
  \label{eq:dual}
  \begin{cases}
    dz = - Az\, dt -a(t) Z\, dt +Z\, dw(t), \quad t\in (0,T),\\
    z(T) = \eta.
  \end{cases}
\end{equation}
For each $\eta\in
L^2(\Omega,\cF_T,\bP;L^2(D))$, it is known (see, for example
\cite{hu1991adapted, du2010revisit}) that the equation \eqref{eq:dual} admits a
unique solution $(z,Z)$ in the space of
$(L^2_\cF(\Omega;C(0,T;L^2(D)))\cap L^2_\cF(0,T;H^1_0(D)))\times L^2_\cF(0,T;L^2(D))$.

We also define a functional on $L^2(\Omega; \cF_T, \bP;L^2(D))$:
\begin{equation}
  \label{eq:J_epsilon}
  \cJ_\varepsilon(\eta) = \frac{1}{2}\int_0^T \E \|z(t;\eta)\|_G^2\, dt + \varepsilon (\E \|\eta\|^2)^{1/2} + \E(y_0, z(0;\eta)).
\end{equation}

\begin{theorem}
  \label{thm:main}
  The system \eqref{eq:main} is approximately controllable at time $T$.
  Moreover, denote by $\eta^\ast$ the minimizer of the functional defined in
  \eqref{eq:J_epsilon}, 
  $u^\ast = z(\cdot;\eta^\ast)\restriction_G$ is a minimal norm control
  and
  \[
    \int_0^T \E \|u^\ast(t)\|_G^2\, dt \le C (\E \|y_0\|^2)^2,
  \]
  where $z(\cdot;\eta)$ is the solution of equation \eqref{eq:dual} with
  $z(T)=\eta^\ast$.
\end{theorem}

Controllability problems of deterministic partial differential
equations are extensive studied in literature; see the survey articles
\cite{zuazua2002controllability, zuazua2007controllability} and references
therein. Recently, there are also some results obtained for the stochastic
counterpart \cite{barbu2003carleman, fu2017controllability, liu2014controllability, lu2011some}, to name a few. In particular, authors in
\cite{lu2011some} pointed out a different observation from the deterministic
equations that null controllability does not necessarily imply approximate
controllability in the stochastic system, and showed the later by duality
argument and a Riesz-type representation theorem for general stochastic
processes \cite{lu2012representation}. Unlike the method employed in \cite{lu2011some}, we use a variational
technique and provide a
constructive proof of approximate controllability, and furthermore, our method
leads to the existence of the minimal norm control.

Thanks to the observability inequality developed in \cite{yang2016observability} (see \cite{phung2013observability, apraiz2014observability} for the deterministic version), our approximate controllability result is more robust and can allow the
control domain $G$ to be any measurable set with positive Lebesgue measure,
compared to works mentioned earlier with the internal control living on an open
set. This generalization facilitates the study of the optimal actuator location problem for
a wider class of equations. For example, authors in
\cite{guo2016optimal} investigate the optimal actuator location of the
minimum norm controls for deterministic heat equations in arbitrary dimensions,
while \cite{allaire2010long} studied the one
dimensional case and \cite{guo2015optimal} considered a special class of
controlled domains. For other actuator location problems, see
\cite{darivandi2013algorithm, privat2015optimal}, and related numerical research
\cite{muench2008optimal, munch2009optimal, tiba2011finite}.

The second part of our paper is devoted to the optimal actuator location of the
minimum norm control problem for internal approximate controllable stochastic
heat equations.

Given $\alpha\in (0,1)$, let
\begin{equation}
  \label{eq:W}
  \cW = \{G\subseteq D \mid G~\text{is Lebesgue measurable with}~
  |G| = \alpha|D|\},
\end{equation}
where  $|\cdot|$ is the Lebesgue measure on $\R^d$.

A classical optimal actuator location of the minimal norm control problem is to seek a set $G^\ast\in \cW$ such that
\begin{equation}
  \label{eq:optimal location}
  N_\e(G^\ast) = \inf_{G\in \cW} N_\e(G).
\end{equation}
If such a $G^\ast$ exists, we say that $G^\ast$ is an {\it optimal actuator location} of the minimum norm controls. Any minimum norm control $u^\ast$ satisfying
\[
\E \|u^\ast\chi_{G^\ast}\|^2_{L^2((0,T)\times D)} = N_\e(G^\ast),
\]
is called a {\it minimum norm control} with respect to the optimal actuator location $G^\ast$.

The existence of the optimal actuator location $G^\ast$ is generally not
guaranteed because of the absence of the compactness of $\cW$. For this reason,
we extend the feasible set $\cW$ to a relaxed set $\cB$ (see \eqref{eq:B}), and
solve a relaxed optimal actuator location problem. We prove the existence of the relaxed optimal actuator location and characterize the solution of the relaxed problem via a Nash equilibrium; see Section 3 for problem formulation and details.

The key contribution of this paper is that we are able to recover the solution from a relaxed
problem to a solution of the classical optimal actuator location problem.

\begin{theorem}
  \label{thm:main2}
  There exists a solution to problem \eqref{eq:optimal location}.
\end{theorem}

To the best of our knowledge, this is the first result on the existence of
optimal actuator location problem in a stochastic system, and it can easily be
applied to the null controllability of stochastic heat equations discussed in
\cite{yang2016observability}. In a related work \cite{privat2015optimal},
authors studied a deterministic heat equation with random initial data, and
minimized the constant in the observability inequality, instead of the norm of
admissible controls.

The rest of the paper is organized as follows. In Section
\ref{sec:2}, we show the existence of the minimizer of $J_\e$ and construct the minimum
norm control to prove Theorem \ref{thm:main}. In Section
\ref{sec:3}, we solve a relaxed optimal actuator location
problem and provide a sufficient and necessary condition for the relaxed optimal
solution via Nash equilibrium. In Section \ref{sec:4}, we show the existence of the a
classical optimal actuator location problem.

\section{Characterization of the minimum norm control}
\label{sec:2}
In this section, we study the approximate controllability of the system
\eqref{eq:main}. We provide a constructive proof of the existence of an
admissible control, and then show that this control is indeed a minimum norm
control that solves problem \eqref{eq:norm optimal}.

First, we prove the existence of the minimizer of $\cJ_\e$ defined in \eqref{eq:J_epsilon}.
\begin{lemma}
  \label{lem:J_min}
  Suppose $y_0\in L^2(\Omega,\cF_0, \bP; L^2(D))$. Then there exists $\eta^\ast\in
  L^2(\Omega, \cF_T, \bP;L^2(D))$ such that
  \[
    \cJ_\varepsilon(\eta^\ast) = \min_{\eta\in L^2(\Omega,\cF_T,\bP;L^2(D))} \cJ_\varepsilon(\eta).
  \]
\end{lemma}
\begin{proof}
  We first show that $\cJ_\varepsilon(\cdot)$ is convex. To this end, fix
  $\theta\in (0,1)$ and let $\eta_1,\eta_2 \in L^2(\Omega,\cF_T, \bP; L^2(D))$.
  By linearity, we have
  \[
    z(\cdot; \theta \eta_1 + (1-\theta)\eta_2) = \theta z(\cdot;\eta_1) + (1-\theta)z(\cdot;\eta_2).
  \]
  Thus, 
  \begin{align*}
   & \cJ_\varepsilon(\theta\eta_1 + (1-\theta)\eta_2)\\
   =\, &  \frac{1}{2}\int_0^T \E \|\theta z(\cdot;\eta_1) + (1-\theta)z(\cdot;\eta_2)\|_G^2\, dt+ \varepsilon(\E \|\theta\eta_1+(1-\theta)\eta_2\|^2)^{1/2} \\
    & \qquad +\E(y_0, \theta z(0;\eta_1) + (1-\theta)z(0;\eta_2))\\
    \le\,& \frac{\theta}{2}\int_0^T \E\|z(\cdot;\eta_1)\|^2_G\, dt + \frac{1-\theta}{2}\int_0^T \E\|z(\cdot;\eta_2)\|^2_G\, dt\\
    & \quad + \varepsilon\theta(\E\|\eta_1\|^2)^{1/2} + \varepsilon(1-\theta)(\E\|\eta_2\|^2)^{1/2}+ \E(y_0, \theta z(0;\eta_1) + (1-\theta)z(0;\eta_2)).
  \end{align*}
  On the other hand,
  \begin{align*}
    & \theta \cJ_\varepsilon(\eta_1) + (1-\theta)\cJ_\varepsilon(\eta_2)\\
    =\ & \theta \left[\frac{1}{2} \int_0^T\E \|z(\cdot;\eta_1)\|_G^2\, dt + \varepsilon\E \|\eta_1\| + \E (y_0, z(0;\eta_1))\right]\\
    & \qquad + (1-\theta)\left[\frac{1}{2} \int_0^T\E \|z(\cdot;\eta_2)\|_G^2\, dt + \varepsilon\E \|\eta_2\| + \E (y_0, z(0;\eta_2))\right]\\
    =\ & \frac{\theta}{2}\int_0^T \E \|z(\cdot;\eta_1)\|_G^2\, dt + \frac{1-\theta}{2}\int_0^T \E\|z(\cdot;\eta_2)\|^2_G\, dt\\
    & \qquad + \varepsilon\theta(\E\|\eta_1\|^2)^{1/2} + \varepsilon(1-\theta)(\E\|\eta_2\|^2)^{1/2} + \E(y_0, \theta z(0;\eta_1) + (1-\theta)z(0;\eta_2)).
  \end{align*}
  Therefore, we obtain
  \begin{align*}
     \cJ_\varepsilon(\theta\eta_1 + (1-\theta)\eta_2)\le   \theta \cJ_\varepsilon(\eta_1) + (1-\theta)\cJ_\varepsilon(\eta_2),
  \end{align*}
  and so $\cJ_\varepsilon(\cdot)$ is convex.

  Next, we prove that the functional $\cJ_\varepsilon(\cdot)$ is continuous. In fact, assume $\eta, \varphi\in L^2(\Omega, \cF_T, \bP; L^2(D))$,
  and $h\in \R$, then
  \begin{align*}
    & \cJ_\e(\eta + h \varphi) - \cJ_\e(\eta)\\
    =\, & \frac{1}{2} \int_0^T \E \|z(t;\eta+h \varphi)\|^2_G\, dt + \e (\E \|\eta+h\varphi\|^2)^{1/2} + \E(y_0, z(0;\eta+h\varphi))\\
    &\quad -\frac{1}{2} \int_0^T \E \|z(t;\eta)\|^2_G\, dt - \e (\E \|\eta\|^2)^{1/2} - \E(y_0, z(0;\eta))\\
    =\, & h \int_0^T\E (z(t;\eta), z(t;\varphi))_G\, dt + \frac{1}{2} h^2 \int_0^T \E \|z(t;\varphi)\|^2_G\, dt\\
    &\quad + \e [(\E \|\eta+h\varphi\|^2)^{1/2}-(E \|\eta\|^2)^{1/2}] + h\E(y_0, z(0;\varphi))\to 0,
  \end{align*}
  as $h\to 0$, since $|\E \|\eta+h\varphi\|^2)^{1/2}-(E \|\eta\|^2)^{1/2}|\le |h|(\E\|\eta\|^2)^{1/2}$.

  To prove the coercivity, let $\{\eta_n\}\subseteq L^2(\Omega,
  \cF_T,\bP;L^2(D))$ such that $\E\|\eta_n\|\to \infty$ as $n\to\infty$, and let
\[
  \wt{\eta}_n = \eta_n/ (\E \|\eta_n\|^2)^{1/2},
\]
so that $(\E \|\wt{\eta}_n\|^2)^{1/2} = 1$. Then
\begin{align*}
  \frac{\cJ_{\e}(\eta_n)
  }{(\E \|\eta_n\|^2)^{1/2}} = \frac{1}{2}(\E \|\eta_n\|^2)^{1/2} \int_0^T \E \|z(t;\wt{\eta}_n)\|^2_G\, dt + \e + \E(y_0, z(0;\wt{\eta}_n)).
\end{align*}
Note that
\[
  |\E(y_0, z(0;\wt{\eta}_n))| \le (\E \|y_0\|^2)^{1/2} (\E
  \|z(0,\wt{\eta}_n)\|^2)^{1/2} \le C(\E \|y_0\|^2)^{1/2}. 
\]

If $\liminf_{n\to\infty}\int_0^T \E \|z(t;\wt{\eta}_n)\|^2_G\, dt>0$, we have
\begin{equation}
  \label{eq:coercivity}
  \cJ_{\e}(\eta_n)
   \to \infty\quad\text{as}\quad (\E\|\eta_n\|^2)^{1/2}\to\infty, 
\end{equation}
which implies the coercivity of $\cJ_\e$.

If $\liminf_{n\to\infty}\int_0^T \E \|z(t;\wt{\eta}_n)\|^2_G\, dt=0$, then it
follows from the observability inequality in \cite[Theorem
1.1]{yang2016observability} that $\wt{\eta}_n$ is bounded in $L^2(\Omega, \cF_T,
\bP; L^2(D))$. Thus, we can extract a subsequence $\{\wt{\eta}_{n_j}\}\subseteq
\{\wt{\eta}_n\}$ such that $\wt{\eta}_{n_j}\rightharpoonup \wt{\eta}$ weakly in
$L^2(D)$ and $z(\cdot; \wt{\eta}_{n_j})\rightharpoonup z(\cdot;\wt{\eta})$ weakly in
$L^2(0,T; H^1_0(D))\cap H^1(0,T;H^{-1}(D)), ~\bP$-a.s. Moreover, by lower
semi-continuity we obtain
\[
  \int_0^T \E \|z(t;\wt{\eta})\|_G^2\, dt \le \liminf_{j\to\infty}\int_0^T \E
  \|z(t; \wt{\eta}_{n_j})\|_G^2\, dt = 0.
\]
Then by the observability inequality again, we get $z(0;\wt{\eta})=0$, and thus
\begin{align*}
  \liminf_{n\to\infty}\frac{\cJ_{\e}(\eta_n)
  }{(\E \|\eta_n\|^2)^{1/2}}
  & \ge \liminf_{j\to\infty}\frac{\cJ_{\e}(\eta_{n_j})
    }{(\E \|\eta_{n_j}\|^2)^{1/2}}\\
  & \ge \liminf_{j\to \infty} [\e + \E(y_0, z(0;\wt{\eta}_{n_j}))]\\
  & = \e + \E(y_0, z(0,\wt{\eta})) = \e,
\end{align*}
which implies~\eqref{eq:coercivity}, and so $\cJ_\e$ is coercive.

To sum up, we showed that $\cJ_\e$ is convex, continuous, and coercive, and thus
the minimizer of $\cJ_\e$ exists..

\end{proof}

Now we are ready to prove Theorem \ref{thm:main}, and characterize the minimum
norm control via the minimizer $\eta^\ast$ of $\cJ_\e$ obtained in Lemma \ref{lem:J_min}.
\begin{proof}[Proof of Theorem \ref{thm:main}]
  Since $\cJ_\e$ attains its minimum value at $\eta^\ast$, for any $\eta\in
  L^2(\Omega, \cF_T, \bP;L^2(D))$ and $h\in \R$ we have
  \[
    \cJ_\e(\eta^\ast) \le \cJ_\e(\eta^\ast+h\eta).
  \]
  Thus,
  \begin{align}
    \label{eq:J-J}
    & \cJ_\e(\eta^\ast + h \eta) - \cJ_\e(\eta^\ast)\\
    =\, &  \frac{1}{2} \int_0^T \E \|z(t;\eta^\ast + h \eta)\|_G^2\, dt + \e (\E \|\eta^\ast + h \eta\|^2)^{1/2} + \E (y_0, z(0;\eta^\ast + h \eta))\notag\\
     & \quad - \frac{1}{2} \int_0^T \E \|z(t;\eta^\ast)\|_G^2\, dt - \e (\E \|\eta^\ast\|^2)^{1/2} - \E (y_0, z(0;\eta^\ast))\notag\\
   =\, &  \frac{h^2}{2} \int_0^T \E \|z(t;\eta) \|_G^2\, dt + h \int_0^T \E (z(t;\eta^\ast), z(t;\eta))_G\, dt\notag\\
    & \quad + \e [(\E \|\eta^\ast + h \eta\|^2)^{1/2}-(\E \|\eta^\ast\|^2)^{1/2}] +h \E(y_0, z(0;\eta))  \ge 0.
  \end{align}
  Also note that
  \[
    |(\E \|\eta^\ast + h \eta\|^2)^{1/2}-(\E \|\eta^\ast\|^2)^{1/2}| \le |h| (\E \|\eta\|^2)^{1/2},
  \]
  then we obtain
  \begin{multline*}
    \frac{h^2}{2} \int_0^T \E \|z(t;\eta) \|_G^2\, dt + h \int_0^T \E (z(t;\eta^\ast), z(t;\eta))_G\, dt
    + \e |h| (\E \|\eta\|^2)^{1/2}\\ +h \E(y_0, z(0;\eta)) \ge 0.
  \end{multline*}
  If $h>0$, then dividing $h$ and sending $h\to 0+$ yield
  \[
\int_0^T \E (z(t;\eta^\ast), z(t;\eta))_G\, dt
    + \e  (\E \|\eta\|^2)^{1/2} +\E(y_0, z(0;\eta)) \ge 0.
  \]
  If $h<0$, then diving $h$ and sending $h\to 0-$ yield
\[
  \int_0^T \E (z(t;\eta^\ast), z(t;\eta))_G\, dt
    - \e  (\E \|\eta\|^2)^{1/2} +\E(y_0, z(0;\eta)) \le 0.
\]
Consequently,
\begin{equation}
  \label{eq:2.145}
  \left| \int_0^T \E (z(t;\eta^\ast), z(t;\eta))_G\, dt
     +\E(y_0, z(0;\eta)) \right| \le \e  (\E \|\eta\|^2)^{1/2}.
\end{equation}

Now applying \Ito's formula to $(y, z)$ yields 
\begin{align*}
  d(y, z)
  & = (y, dz) + (dy, z) + a(y, Z)\, dt\\
  & = -(y, Az)\, dt - a(y, Z)\, dt + (y, Z)\, dw(t)\\
  & \quad +(Ay, z)\, dt + (u, z)_G \, dt + a(y, z)\, dw(t) + a(y, Z)\, dt\\
  & = (u, z)_G\, dt + (y, Z)\, dw(t) + a(y, z)\, dw(t).
\end{align*}
After we rewrite it in the integral form and take the expectation on both sides,
we get
\begin{equation}
  \label{eq:ito}
    \E(y_0, z(0;\eta)) = \E(y(T;u), \eta) - \int_0^T \E(z(t;\eta), u).
\end{equation}
Plugging~(\ref{eq:ito}) into~(\ref{eq:2.145}) and letting $u= u^\ast=
z(\cdot;\eta^\ast)\restriction_G$, we obtain
\[
  |\E(y(T;u^\ast), \eta)| \le \e (\E\|\eta\|^2)^{1/2}.
\]
Since $\eta$ is an arbitrary element in $L^2(\Omega, \cF_T, \bP;L^2(D))$, we
conclude that
\[
  (\E \|y(T;u^\ast)\|^2)^{1/2} \le \e.
\]
Therefore, $u^\ast = z(\cdot;\eta^\ast)\restriction_G$ is an admissible control.

Next, we show that $u^\ast$ is indeed the minimum norm control. To this end, let's go
back to the inequality~(\ref{eq:J-J}). Note that
\begin{align*}
  & \lim_{h\to 0} \frac{(\E\|\eta^\ast + h\eta\|^2)^{1/2} - (\E \|\eta^\ast\|^2)^{1/2}}{h}\\
  =\, & \lim_{h\to 0} \frac{\E \|\eta^\ast + h\eta\|^2 - \E \|\eta^\ast\|^2}{h[(\E\|\eta^\ast + h\eta\|^2)^{1/2} + (\E \|\eta^\ast\|^2)^{1/2})]}\\
  =\, & \frac{1}{2(\E \|\eta^\ast\|^2)^{1/2}}\lim_{h\to 0} \frac{\E \|\eta^\ast + h\eta\|^2 - \E \|\eta^\ast\|^2}{h}\\
  =\, & \frac{1}{2(\E \|\eta^\ast\|^2)^{1/2}}\lim_{h\to 0} \frac{2h \E(\eta^\ast, \eta) + h^2 \E\|\eta\|^2}{h}\\
  =\, & \frac{\E(\eta^\ast, \eta)}{(\E\|\eta^\ast\|^2)^{1/2}}.
\end{align*}
If $h>0$, then dividing $h$ and sending $h\to 0+$ in \eqref{eq:J-J} yields
\[
  \int_0^T \E(z(t;\eta^\ast), z(t;\eta))_G\, dt + \e \cdot \frac{\E(\eta^\ast,
    \eta)}{(\E\|\eta^\ast\|^2)^{1/2}} + \E(y_0, z(0;\eta)) \ge 0.
\]
If $h<0$, then dividing $h$ and sending $h\to 0-$ in \eqref{eq:J-J} yields
\[
   \int_0^T \E(z(t;\eta^\ast), z(t;\eta))_G\, dt + \e \cdot \frac{\E(\eta^\ast,
    \eta)}{(\E\|\eta^\ast\|^2)^{1/2}} + \E(y_0, z(0;\eta)) \le 0.
\]
To sum up, we have the following Euler-Lagrange equation
\begin{equation}
  \label{eq:9}
    \int_0^T \E(z(t;\eta^\ast), z(t;\eta))_G\, dt + \e \cdot \frac{\E(\eta^\ast,
    \eta)}{(\E\|\eta^\ast\|^2)^{1/2}} + \E(y_0, z(0;\eta)) =0.
\end{equation}
Plugging~(\ref{eq:ito}) into~(\ref{eq:9}) and using $u^\ast =
z(\cdot;\eta^\ast)\restriction_G$, we get
\[
  \E(y(T;u^\ast), \eta) + \e \cdot \frac{\E(\eta^\ast,
    \eta)}{(\E\|\eta^\ast\|^2)^{1/2}} =0.
\]
If particular, by choosing $\eta = \eta^\ast$,
\begin{equation}
  \label{eq:9-1}
  \E(y(T;u^\ast), \eta^\ast) = -\e (\E\|\eta^\ast\|^2)^{1/2}.
\end{equation}

Suppose now there is another admissible control $\wt{u}$ such that
\[
  (\E\|y(T;\eta^\ast)\|^2)^{1/2} \le \e.
\]
Then we have by \eqref{eq:9-1}
\begin{align*}
  \E(y(T;\wt{u}), \eta^\ast) & \ge - (\E\|y(T;\wt{u})\|^2)^{1/2} (\E\|\eta^\ast\|^2)^{1/2}\\
  & \ge -\e (\E\|\eta^\ast\|^2)^{1/2} = \E(y(T;u^\ast), \eta^\ast).
\end{align*}
Using~(\ref{eq:ito}) again, we arrive at
\[
  \int_0^T \E(z(t;\eta^\ast), \wt{u})_G\, dt \ge \int_0^T \E (z(t;\eta^\ast), u^\ast)_G\, dt
\]
Since $u^\ast = z(\cdot;\eta^\ast)\restriction_G$, we obtain
\[
  \int_0^T \E \|u^\ast
  (t)\|_G^2\, dt \le \int_0^T \E(u^\ast(t), \wt{u}(t))_G\, dt,
\]
and by Cauchy-Schwarz inequality,
\[
  \int_0^T \E \|u^\ast
  (t)\|_G^2\, dt \le \int_0^T \E \|\wt{u}
  (t)\|_G^2\, dt,
\]
which implies the optimality of $u^\ast$.

To show the boundedness for $u^\ast$, let's replace $\eta$ in \eqref{eq:9}, and
then
\begin{equation*}
       \int_0^T \E \|z(t;\eta^\ast)\|_G^2\, dt + \e(\E\|\eta^\ast\|^2)^{1/2} + \E(y_0, z(0;\eta^\ast))=0.
\end{equation*}
By the observability inequality (see \cite[Lemma
3.2]{yang2016observability}) for the (relaxed) system~\eqref{eq:relaxed eqn} that 
\[
  \E \|z(0;\eta^\ast)\|^2 \le C \int_0^T \E
  \|z(t;\eta^\ast)\|_G^2\, dt,
\]
for some $C>0$, independent of $\beta$, and thus
\begin{equation}
  \label{eq:0924-2}
  \frac{1}{C} \E \|z(0;\eta^\ast)\|^2 + \e(\E\|\eta^\ast\|^2)^{1/2} + \E(y_0,
  z(0;\eta^\ast))\le 0.
\end{equation}
Consequently,
\[
  \frac{1}{C} \E \|z(0;\eta^\ast)\|^2\le -  \E(y_0,
  z(0;\eta^\ast))\le (\E\|y_0\|^2)^{1/2} (\E \|z(0;\eta^\ast)\|^2)^{1/2},
\]
or
\[
  (\E \|z(0;\eta^\ast)\|^2)^{1/2} \le C (\E\|y_0\|^2)^{1/2}.
\]
Using \eqref{eq:0924-2} again, we have
\[
  \e(\E\|\eta^\ast\|^2)^{1/2} \le (\E\|y_0\|^2)^{1/2} (\E
  \|z(0;\eta^\ast)\|^2)^{1/2} \le C \E \|y_0\|^2,
\]
or equivalently,
\[
  (\E\|\eta^\ast\|^2)^{1/2}  \le \frac{C}{\e} \E \|y_0\|^2.
\]
Therefore,
\begin{align*}
  \int_0^T \E \|u^\ast(t)\|_G^2\, dt
  &  = \int_0^T \E \| z(t;\eta^\ast)\|_G^2\, dt
   \le \int_0^T \E \| z(t;\eta^\ast)\|^2\, dt\\
  & \le C \E \|\eta^\ast\|^2
  \le C (\E \|y_0\|^2)^2,
\end{align*}
which completes the proof.
\end{proof}

\section{A relaxed optimal actuator location problem}
\label{sec:3}
Without the compactness of $\cW$ defined in \eqref{eq:W}, it seems very
difficult to solve the classical optimal actuator location problem directly.
Instead, we study a relaxed problem and provide a solution in the framework of a
two-person zero sum game via Nash equilibrium. To this end, define
\begin{equation}
  \label{eq:B}
  \cB = \left\{\beta \in L^\infty(D;[0,1]) \mid \|\beta\|^2 = \alpha |D|\right\}.
\end{equation}
Note that the set $\cB$ is a relaxation of the set $\{\chi_G\mid G\in
\cW\}$.

For any $\beta\in \cB$, consider the following equation
\begin{equation}
  \label{eq:relaxed eqn}
  \begin{cases}
d y = A y\, dt + 
\beta u(t)\, dt + a(t) y\, dw(t),
& t\in (0,T),\\
y(0) = y_0.
\end{cases}
\end{equation}
We denote by $y(\cdot;\beta,u)$ the solution of equation
\eqref{eq:relaxed eqn}, and say the system \eqref{eq:relaxed eqn} {\it
  approximately controllable} at time $T$ if for any $\e>0$, there exists $u\in L^2_\cF(0,T;L^2(D))$ such
that $\E \|y(T;\beta, u)\|^2 \le \e$. Accordingly, the problem
\eqref{eq:norm optimal} is replaced by 
\begin{equation}
  \label{eq:norm optimal relax}
  N_\e(\beta) = \inf \{ \E \|u\|^2_{L^2((0,T)\times D)} \mid \E \|y(T;\beta,
  u)\|^2\le \e \},
\end{equation}
and the classical optimal actuator location problem \eqref{eq:optimal
  location} is changed into the following relaxed problem
\begin{equation}
  \label{eq:optimal location relax}
  N_\e(\beta^\ast) = \inf_{\beta\in\cB} N_\e(\beta).
\end{equation}
Any solution $\beta^\ast$ to the problem \eqref{eq:optimal location
  relax} is called a {\it relaxed optimal actuator location} of the
minimal norm controls.

\begin{theorem} 
  \label{thm:nash}
  There exists at least one solution of the problem \eqref{eq:optimal location
    relax}. In addition, $\beta^\ast$ is a relaxed optimal actuator
  location of the  minimal norm controls if and only if there
  exists $\eta^\ast\in \cU_M$ such that the pair
  $(\beta^\ast,\eta^\ast)$ is a Nash equilibrium
  of the following two-person zero-sum game problem: to find
  $(\beta^\ast,\eta^\ast)\in\cB\times \cU_M$ such that
\begin{align*}
& \frac{1}{2}\int_0^T \E \|\beta^\ast z(t;\eta^\ast)\|^2 dt + \e(\E \|\eta^\ast\|^2)^{1/2}+
  \E(y_0, z(0;\eta^\ast))\\
  &\quad = 
\sup_{\beta\in\cB}\left[\frac{1}{2}\int_0^T \E \|\beta z(t;\eta^\ast)\|^2 dt + \e(\E \|\eta^\ast\|^2)^{1/2}+
  \E(y_0, z(0;\eta^\ast))\right],\\
 & \frac{1}{2}\int_0^T \E \|\beta^\ast z(t;\eta^\ast)\|^2 dt + \e(\E \|\eta^\ast\|^2)^{1/2}+
  \E(y_0, z(0;\eta^\ast))\\
  &\quad = 
\inf_{\eta\in\cU_M}\left[\frac{1}{2}\int_0^T \E \|\beta^\ast z(t;\eta)\|^2 dt + \e(\E \|\eta\|^2)^{1/2}+
  \E(y_0, z(0;\eta))\right],
\end{align*}
where $\cU_M = \{\eta \in L^2(\Omega,\cF_T,\bP;L^2(D))\mid \E \|\eta\|^2\le
M\}$, for some constant $M>0$ (see Remark \ref{rem:1}).
\end{theorem}

To prove this theorem, we first study a
variational problem
\begin{equation}
  \label{eq:J-beta}
  \cJ_\e(\beta) = \inf_{\eta\in L^2(\Omega,\cF_T,\bP;L^2(D))} \cJ_\e(\eta;\beta),
\end{equation}
where 
\begin{equation}
  \label{eq:J-z-beta}
  \cJ_\e(\eta;\beta) = \frac{1}{2}\int_0^T \E \|\beta z(t;\eta)\|^2 dt + \e(\E \|\eta\|^2)^{1/2}+
  \E(y_0, z(0;\eta)).
\end{equation}
\begin{lemma}
  \label{lem:3.6}
  Fix $\beta\in \cB$. Suppose $y_0\in L^2(\Omega,\cF_0,\bP;L^2(D))$. Then
  \begin{enumerate}
  \item[(1)] Problem \eqref{eq:J-beta} admits a unique
    solution $\eta^\ast$ in the space $L^2(\Omega,\cF_T,\bP;L^2(D))$, and there
    exists a positive constant $C$, independent of $\beta$ such that
    \begin{equation}
      \label{eq:bound-eta}
    (\E \|\eta^\ast\|^2)^{1/2} \le C \E \|y_0\|^2.  
    \end{equation}
    \item[(2)] The control defined by
\begin{equation}
  \label{eq:u star}
  u^\ast = \beta z(\cdot;\eta^\ast)
\end{equation}
is the minimal norm control to the problem \eqref{eq:norm
  optimal relax}, where $z(\cdot;\eta^\ast)$ is the solution of the adjoint
system \eqref{eq:dual} with $z(T) = \eta^\ast$.
Moreover,  
\begin{equation}
  \label{eq:u star norm}
  N_\e(\beta) = - 2 \cJ_\e(\beta).
\end{equation}
  \end{enumerate}
\end{lemma}
\begin{proof}
In the same spirit of Lemma~\ref{lem:J_min} and Theorem~\ref{thm:main}, the
existence of $\eta^\ast$ and optimality of $u^\ast$ can be similarly verified.

Moreover, we can also obtain the following Euler-Lagrange equation
\begin{equation*}
\int_0^T \E(\beta^2z(t;\eta^\ast), z(t;\eta))\, dt + \e \cdot \frac{\E(\eta^\ast,
    \eta)}{(\E\|\eta^\ast\|^2)^{1/2}} + \E(y_0, z(0;\eta)) =0,
\end{equation*}
for any $\eta\in L^2(\Omega,\cF_T,\bP;L^2(D))$. In particular, choosing
$\eta=\eta^\ast$ yields
\begin{equation}
  \label{eq:0917-3}
       \int_0^T \E \|\beta z(t;\eta^\ast)\|^2\, dt + \e(\E\|\eta^\ast\|^2)^{1/2} + \E(y_0, z(0;\eta^\ast))=0.
\end{equation}
It follows from the observability inequality (see \cite[Lemma
3.2]{yang2016observability}) for the (relaxed) system~\eqref{eq:relaxed eqn} that 
\[
  \frac{1}{C} \E \|z(0;\eta^\ast)\|^2 \le \int_0^T \E
  \|\beta z(t;\eta^\ast)\|^2\, dt,
\]
for some $C>0$, independent of $\beta$, and thus
\begin{equation}
  \label{eq:0917-2}
  \frac{1}{C} \E \|z(0;\eta^\ast)\|^2 + \e(\E\|\eta^\ast\|^2)^{1/2} + \E(y_0,
  z(0;\eta^\ast))\le 0.
\end{equation}
Consequently,
\[
  \frac{1}{C} \E \|z(0;\eta^\ast)\|^2\le -  \E(y_0,
  z(0;\eta^\ast))\le (\E\|y_0\|^2)^{1/2} (\E \|z(0;\eta^\ast)\|^2)^{1/2},
\]
or
\[
  (\E \|z(0;\eta^\ast)\|^2)^{1/2} \le C (\E\|y_0\|^2)^{1/2}.
\]
Using \eqref{eq:0917-2} again, we have
\[
  \e(\E\|\eta^\ast\|^2)^{1/2} \le (\E\|y_0\|^2)^{1/2} (\E
  \|z(0;\eta^\ast)\|^2)^{1/2} \le C \E \|y_0\|^2,
\]
or equivalently,
\[
  (\E\|\eta^\ast\|^2)^{1/2}  \le \frac{C}{\e} \E \|y_0\|^2,
\]
which proves~\eqref{eq:bound-eta}. 

Setting $z(\cdot;\eta^\ast) = u^\ast/\beta$ in \eqref{eq:0917-3} yield
\begin{equation}
  \label{eq:0917}
    \int_0^T \E \|u^\ast(t)\|^2\, dt = - \e(\E\|\eta^\ast\|^2)^{1/2} - \E(y_0, z(0;\eta^\ast)),
\end{equation}
or equivalently
\[
  N_\e(\beta) = - \e(\E\|\eta^\ast\|^2)^{1/2} - \E(y_0, z(0;\eta^\ast)).
\]
On the other hand, we have by \eqref{eq:0917}
\begin{align*}
  \cJ_\e(\beta)
  & = \cJ_\e(\eta^\ast;\beta) = 
    \frac{1}{2}\int_0^T \E \|\beta z(t;\eta^\ast)\|^2 dt + \e(\E \|\eta^\ast\|^2)^{1/2}+  \E(y_0, z(0;\eta^\ast))\\
  &= \frac{1}{2}\int_0^T \E \|u^\ast(t)\|^2 dt + \e(\E \|\eta^\ast\|^2)^{1/2}+  \E(y_0, z(0;\eta^\ast))\\
  &= \frac{1}{2} \left[ \e(\E \|\eta^\ast\|^2)^{1/2}+  \E(y_0, z(0;\eta^\ast))\right]\\
  & = - \frac{1}{2} N_\e(\beta),
\end{align*}
which concludes the proof.
\end{proof}
\begin{remark}
  \label{rem:1}
  It follows from \eqref{eq:bound-eta} that we can restrict our search of the
  minimizer of $\cJ_\e(\cdot;\beta)$ within a closed ball of $L^2(\Omega,
  \cF_T,\bP;L^2(D))$, and we denote it by $\cU_M$ for some $M>0$.

\end{remark}

Now let's formulate an equivalent problem to \eqref{eq:optimal location relax},
which is a two-person zero sum game. Define
\begin{equation}
  \label{eq:theta}
  \Theta = \left\{\theta\in L^\infty(D;[0,1]) \mid \int_D \theta(x) dx
      = \alpha |D|\right\}.
\end{equation}
It is clear that
\begin{equation}
  \label{eq:theta-beta}
  \beta^2\in \Theta~\text{for any}~\beta\in\cB,
  ~\text{and}~\theta^{1/2}\in\cB~\text{for all}~\theta\in\Theta.
\end{equation}
We also define a functional $f: \Theta \times \cU_M \to \R$ by
\begin{equation}
  \label{eq:F}
  f(\theta, \eta) = -\frac{1}{2}\int_0^T \E \|\theta^{1/2} z(t;\eta)\|^2 dt - \e(\E \|\eta\|^2)^{1/2}-  \E(y_0, z(0;\eta)).
\end{equation}
Then it follows from the relation~\eqref{eq:u star norm} that
\begin{align*}
  & \inf_{\beta\in \cB}\frac{1}{2}N_\e(\beta) = \inf_{\beta\in \cB} -\cJ_\e(\beta)\\
  =\, & \inf_{\theta \in \Theta} \sup_{\eta\in \cU_M} f(\theta, \eta).
\end{align*}
Therefore, seeking a minimizer $\beta^\ast\in\cB$ for $N(\beta)$ amounts to
finding a minimizer $\theta^\ast$ for $\sup_{\eta\in \cU_M} f(\theta, \eta)$.

\begin{proof}[Proof of Theorem \ref{thm:nash}]
  In the following, we will solve both the problems
  \[
    \inf_{\theta \in \Theta} \sup_{\eta\in \cU_M} f(\theta,
    \eta)\quad\text{and}\quad \sup_{\eta\in \cU_M}\inf_{\theta \in \Theta}  f(\theta, \eta),
  \]
  at the same time.
\begin{enumerate}
\item[(1)] Let us equip $L^\infty(D)$ with the weak$^\ast$ topology. Then $\Theta$ is
compact and convex in $L^\infty(D)$.
\item[(2)] Since $\cU_M$ is a closed ball, it is weak
compact and convex in $L^2(\Omega,\cF_T,\bP;L^2(D))$.
\item[(3)] Apparently, for each $\eta\in \cU_M$, the function
\[
  \theta \mapsto f(\theta, \eta)
\]
is convex and continuous.
\item[(4)] Fix $\theta\in \Theta$, assume $\eta_n$ converges weakly in $\cU_M$
  to $\eta$. Then $z(\cdot;\eta_n)$ converges weakly to $z(\cdot;\eta)$. In
  other words, the function
  \[
    \eta \mapsto f(\theta, \eta)
  \]
  is continuous. It is also easy to see that the function is concave.
\end{enumerate}
It follows from (1)-(4), and the Von Neumann minimax theorem (see \cite[Theorem
2.7.2]{aubin2011applied}) that there exist $\theta^\ast\in \Theta$ and $\eta\in
\cU_M$ that
\[
  f(\theta^\ast, \eta) \le f(\theta^\ast, \eta^\ast) \le f(\theta, \eta^\ast),\quad
  \forall \theta\in \Theta, \eta\in \cU_M.
\]
Note that $\cJ_\e(\eta;\beta) = - f(\theta, \eta)$, where $\theta = \beta^2$,
and thus
\[
  \cJ_\e(\eta^\ast;\beta) \le \cJ_\e(\eta^\ast;\beta^\ast) \le
  \cJ_\e(\eta;\beta^\ast),\quad \forall \beta\in \cB, \eta\in \cU_M,
\]
where $\beta^\ast = (\theta^\ast)^{1/2}$.

By definition, the pair $(\beta^\ast, \eta^\ast)$ is a Nash equilibrium of the
two-person zero sum game problem: to find $(\beta^\ast, \eta^\ast)\in \cB\times
\cU_M$ such that
\[
  \cJ_\e(\eta^\ast;\beta^\ast) = \sup_{\beta\in \cB} \cJ_\e(\eta^\ast;\beta),
\]
and
\[
  \cJ_\e(\eta^\ast;\beta^\ast) = \inf_{\theta\in \cU_M} \cJ_\e(\eta;\beta^\ast).
\]
This completes the proof.
\end{proof}

\section{Existence of classical optimal actuator location}
\label{sec:4}
In this section, we will prove our main result on the existence of classical
optimal actuator location, that is, we solve the problem \eqref{eq:optimal
  location}. To this end, let's take a closer look at the corresponding relaxed
problem. It follows from Theorem \ref{thm:nash} that
\begin{align*}
  \cJ_\e(\eta^\ast;\beta^\ast)
  & = \sup_{\beta\in \cB}\cJ_\e(\eta^\ast;\beta)\\
  & = \sup_{\beta\in \cB}\left[\frac{1}{2}\int_0^T \E \|\beta z(t;\eta^\ast)\|^2 dt + \e(\E \|\eta^\ast\|^2)^{1/2}+
    \E(y_0, z(0;\eta^\ast))\right]\\
  & = \frac{1}{2}\sup_{\beta\in \cB}\int_0^T \E \|\beta z(t;\eta^\ast)\|^2 dt + \frac{1}{2}\left[\e(\E \|\eta^\ast\|^2)^{1/2}+
    \E(y_0, z(0;\eta^\ast))\right]\\
  & = \frac{1}{2}\sup_{\theta\in \Theta}\int_D\theta(x)\int_0^T \E |z(t;\eta^\ast)|^2\, dt\, dx+ \frac{1}{2}\left[\e(\E \|\eta^\ast\|^2)^{1/2}+
    \E(y_0, z(0;\eta^\ast))\right].
\end{align*}
To wit, if $\beta^\ast$ is a relaxed optimal actuator location, then
$\theta^\ast = (\beta^\ast)^2$ solves the the following problem
\[
  \sup_{\theta\in \Theta}\int_D\theta(x)\int_0^T \E |z(t;\eta^\ast)|^2\, dt\, dx.
\]
Let
\[
  H(x) := \int_0^T \E |z(t,x;\eta^\ast)|^2\, dt,
\]
then $H(x)\ge 0$ for $x\in D$-a.e.

Next, we show that
\[
  \theta^\ast = \text{argmax}_{\theta\in \Theta} \int_D \theta(x) H(x)\, dx
\]
is an indicator function. Therefore, $\beta^\ast = (\theta^\ast)^{1/2}$ is also
an indicator function, and thus $\beta^\ast \in \cB$, i.e., the relaxed optimal
actuator location is indeed a classical optimal actuator location.

Define
\[
  c(\alpha) = \sup\{c\ge 0: |\{H\ge c\}| \ge \alpha |D|\}.
\]
Let $\{c_n\}$ be an increasing sequence that converges to $c(\alpha)$. Then for
any $n\in \N$, we have $|\{H\ge c_n\}|\ge \alpha|D|$. But
\[
  \{H\ge c(\alpha)\} = \bigcap_{n=1}^\infty\{H\ge c_n\},
\]
and thus
\begin{equation}
  \label{eq:0926-1}
  |\{H\ge c(\alpha)\}| = \lim_{n\to\infty} |\{H\ge c_n\}| \ge \alpha |D|.
\end{equation}
On the other hand, for any $\e>0$, we have
\[
  |\{H\ge c(\alpha)+\e\}|< \alpha|D|.
\]
Note that
\[
  \{H> c(\alpha)\} = \bigcup_{\e>0}\{H\ge c(\alpha) + \e\},
\]
and so
\begin{equation}
  \label{eq:0926-2}
  |\{H> c(\alpha)\}| = \lim_{\e\to 0} \{H\ge c(\alpha)+\e\} \le \alpha |D|.
\end{equation}
We claim that 
\begin{equation}
    \label{eq:0926-3}
\begin{aligned}
  \theta^\ast(x) & = 1\, \text{for}~x\in \{H>c(\alpha)\}\, \text{a.e.}\\
  \theta^\ast(x) & = 0\, \text{for}~x\in \{H<c(\alpha)\}\, \text{a.e.}
\end{aligned}
\end{equation}
Otherwise, let
\[
  A = \{H>c(\alpha)\}\cap \{\theta^\ast <1\}\quad\text{and}\quad B = \{H<
  c(\alpha)\}\cap \{\theta^\ast > 0\},
\]
then
\[
  |A| >0\quad\text{or}\quad |B|>0.
\]

Let's first assume that $|A|>0$. In this case, we define
\begin{equation}
  \label{eq:0927-1}
    \wt{\theta}(x) =
  \begin{cases}
    1, & \{H>c(\alpha)\}\cup E;\\
    0, &\text{otherwise},
  \end{cases}
\end{equation}
where
\[
  E \subseteq \{H=c(\alpha)\}\quad\text{and}\quad |E| = \alpha|D| - |\{H>c(\alpha)\}|.
\]
It is noted that $\wt{\theta}$ is well defined by \eqref{eq:0926-1} and
\eqref{eq:0926-2}, and $\wt{\theta}\in \Theta$. Thus
\[
  \int_{\theta^\ast>\wt{\theta}}(\theta^\ast-\wt{\theta})\, dx +
  \int_{\theta^\ast< \wt{\theta}}(\theta^\ast-\wt{\theta})\, dx = \int_D
  (\theta^\ast-\wt{\theta})\, dx =0.
\]
Since
\begin{align*}
  \{\wt{\theta}>\theta^\ast\} & = \{\wt{\theta}=1\}\cap \{\theta^\ast <1\} = \left(\{H>c(\alpha)\} \cup E\right) \cap \{\theta^\ast <1\} \\
  & = A \cup \left(E\cap \{\theta^\ast <1\}\right),
\end{align*}
we obtain
\begin{align*}
  \int_{\theta^\ast< \wt{\theta}} (\theta^\ast - \wt{\theta}) H(x)\, dx
  & = \int_{\theta^\ast< \wt{\theta}}\chi_{A \cup \left(E\cap \{\theta^\ast <1\}\right)} (\theta^\ast - \wt{\theta}) H(x)\, dx\\
  & = \int_{\theta^\ast< \wt{\theta}}\chi_A (\theta^\ast - \wt{\theta}) H(x)\, dx \\
  &\qquad + \int_{\theta^\ast< \wt{\theta}}\chi_{E\cap \{\theta^\ast <1\}}(\theta^\ast - \wt{\theta}) H(x)\, dx\\
  & = \int_{\theta^\ast< \wt{\theta}}\chi_A (\theta^\ast - \wt{\theta}) H(x)\, dx \\
  &\qquad + \int_{\theta^\ast< \wt{\theta}}\chi_{E\cap \{\theta^\ast <1\}}(\theta^\ast - \wt{\theta}) c(\alpha)\, dx\\
  & < \int_{\theta^\ast< \wt{\theta}}\chi_A (\theta^\ast - \wt{\theta}) c(\alpha)\, dx \\
  & \qquad +  \int_{\theta^\ast< \wt{\theta}}\chi_{E\cap \{\theta^\ast <1\}}(\theta^\ast - \wt{\theta}) c(\alpha)\, dx\\
  & = c(\alpha)\int_{\theta^\ast< \wt{\theta}}(\chi_A + \chi_{E\cap \{\theta^\ast <1\}}) (\theta^\ast - \wt{\theta}) \, dx\\
  & =  c(\alpha) \int_{\theta^\ast< \wt{\theta}}(\theta^\ast - \wt{\theta}) \, dx.
\end{align*}
Also, we have
\[
  \{\wt{\theta}<\theta^\ast\} \subseteq \{\wt{\theta}>\theta^\ast\}^c \subseteq
  A^c\subseteq \{H\le c(\alpha)\},
\]
and thus
\begin{align*}
  \int_{\theta^\ast>\wt{\theta}}(\theta^\ast-\wt{\theta}) H(x)\, dx
  & = \int_{\theta^\ast>\wt{\theta}}\chi_{\{H\le c(\alpha)\}}(\theta^\ast-\wt{\theta}) H(x)\, dx\\
  & \le c(\alpha) \int_{\theta^\ast>\wt{\theta}}(\theta^\ast-\wt{\theta}) \, dx.
\end{align*}
Consequently,
\begin{align*}
  \int_D (\theta^\ast-\wt{\theta})H(x)\, dx
  & = \int_{\theta^\ast>\wt{\theta}}(\theta^\ast-\wt{\theta}) H(x)\, dx + \int_{\theta^\ast<\wt{\theta}}(\theta^\ast-\wt{\theta}) H(x)\, dx\\
    & < c(\alpha)\left[\int_{\theta^\ast>\wt{\theta}}(\theta^\ast-\wt{\theta})\, dx +
  \int_{\theta^\ast< \wt{\theta}}(\theta^\ast-\wt{\theta})\, dx\right]\\
    & = 0, 
\end{align*}
or
\begin{equation}
  \label{eq:0927-2}
  \int_D \theta^\ast(x)H(x)\, dx < \int_D \wt{\theta}(x)H(x)\, dx
\end{equation}
which contradicts with the fact that
\[
  \int_D \theta^\ast(x)H(x)\, dx = \sup_{\theta\in \Theta}\int_D \theta(x)H(x)\, dx.
\]

In the case that $|B|>0$, we use the same $\wt{\theta}$ defined in
\eqref{eq:0927-1}, and note that
\begin{align*}
  \{\theta^\ast> \wt{\theta}\}
  & = \{\theta^\ast > 0 \}\cap \{\wt{\theta} =0\}\\
  &=   \{\theta^\ast > 0 \}\cap \{H\le c(\alpha)\} \cap E^c\\
  &   = \{\theta^\ast > 0 \}\cap \{H< c(\alpha)\}\\
  & = B.
\end{align*}
By the similar argument, we also get the (strict) inequality \eqref{eq:0927-2},
and thus a contradiction.

To sum up, if $\beta^\ast$ is a relaxed optimal actuator location, then
\[
  \theta^\ast = (\beta^\ast)^2 = \chi_{\{H>c(\alpha)\}} + \chi_E,
\]
where $E \subseteq \{H = c(\alpha)\}$ and $|E| = \alpha|D| - |\{H>c(\alpha)\}|$.

Therefore, by Theorem \ref{thm:nash} we find an optimal actuator location $\beta^\ast$ to our
classical problem \eqref{eq:optimal location}, and thus we finished the proof of
Theorem \ref{thm:main2}.

\bibliographystyle{abbrv}

\end{document}